\documentclass[11pt,oneside, english]{article}
\usepackage{dsfont}
\usepackage{amsmath, amssymb, amsthm, verbatim,enumerate, bbm}
\usepackage{enumerate}
\usepackage[utf8x]{inputenc}
\usepackage[T1]{fontenc}
\usepackage[english]{babel}
\usepackage{mathtools}
\usepackage[
letterpaper,
top=0.8in,
bottom=0.8in,
left=1in,
right=1in]{geometry}
\usepackage{bm} %
\usepackage[colorlinks=true, allcolors=blue]{hyperref}

\usepackage{graphicx}
\graphicspath{{./figures/}}
\usepackage{epsfig}
\usepackage{fullpage}
\usepackage{mathrsfs}
\usepackage{enumerate}
\usepackage{verbatim}
\usepackage{mhequ}
\usepackage{algorithm}
\usepackage{algorithmicx}
\usepackage{algpseudocode}
\usepackage[nameinlink,capitalise,noabbrev]{cleveref}

\usepackage{showlabels}
\numberwithin{equation}{section}

\def \be{\begin{equs}}
\def \ee{\end{equs}}

\newtheorem{theorem}{Theorem}[section]

\newtheorem*{namedtheorem}{\theoremname}
\newcommand{\theoremname}{testing}

\newtheorem*{question*}{Question}

\theoremstyle{definition}

\newcommand{\indep}{\rotatebox[origin=c]{90}{$\models$}}
\theoremstyle{plain}

\title{
 Stochastic Intervention
}

\author{
    Rohit Chaudhuri\thanks{Email: {\tt  rohitchau94\_r@isical.ac.in}.}}

\date{}

\begin{document}

\pagestyle{empty}

\maketitle
\thispagestyle{empty} %

\begin{abstract}
This article discusses the application of stochastic intervention to find the optimal treatment distribution yielding a high value of expected potential outcome under the setting where the number of treatments is allowed to vary with $n$. The primary motivation is to obtain a novel summarization of the effect of various treatments which would guide practitioners towards better decision regarding which intervention to choose.

\end{abstract}


%

  %


\pagestyle{plain}
\setcounter{page}{1}

\section{Problem set-up}
We have $n$ units labelled $i=1,2,\dots,n$ and $K$ binary factorial treatments applied to each unit and is labelled $\boldsymbol{T}_{i}$ with $T_{ik}\in \{0,1\}$ for each $i=1,2,\dots,n$ and $k=1,2,\dots,K$. Let $Y_{i}(\boldsymbol{t})$ denote the potential outcome corresponding to the treatment combination $\boldsymbol{t}$ for the unit $i$, where $\boldsymbol{t}\in \{0,1\}^{K}.$ We work under the following assumptions:
\begin{itemize}
    \item \textit{Randomization} which states that, $$Y_{i}(\boldsymbol{t})\indep \boldsymbol{T}_{i},  $$ for all $\boldsymbol{t}.$
    \item \textit{Positivity} which states that probability of treatment allocation satisfies, $$\mathbb{P}\left(\boldsymbol{T}_{i}=\boldsymbol{t}\right)>0. $$for all $\boldsymbol{t}.$
\end{itemize}
In this article, we assume $\mathbb{P}\left(\boldsymbol{T}_{i}=\boldsymbol{t}\right)=2^{-K}$ for all treatment combination $\boldsymbol{t}$ however our theoretical developments encompass the general case as well. We define the \textit{Expected Potential outcome} corresponding to each treatment combination $\boldsymbol{t}$ to be $c_{\boldsymbol{t}}:=\mathbb{E}\left[Y_{i}(\boldsymbol{t})\right].$ To carry forward the stochastic intervention scheme, we introduce a \textit{Parametric model} over the space of all $2^K$ treatment combinations to be, $$\left\{\mathbb{P}_{\theta}\left(\boldsymbol{T}_{i}=\boldsymbol{t}\right),\boldsymbol{t}\in \{0,1\}^{K},\theta \in \Theta_{K}\right\}. $$ A common example of this would be when we take an independent binomial model across treatment factors in which case the parameter space reduces to $\Theta_{K}=(0,1)^{K}$ and the probabilities are characterized as $$ \mathbb{P}_{\theta}\left(\boldsymbol{T}_{i}=\boldsymbol{t}\right)=\prod_{k=1}^{K}\left\{\theta_{i}^{t_k}(1-\theta_{i})^{1-t_{k}}\right\},$$ for $\boldsymbol{t}=(t_1,\dots,t_{k})$ and  $\boldsymbol{\theta}=(\theta_1,\dots,\theta_{k}) \in \Theta_{K}.$ Hence, the complete parameter space of our problem depicting all the unknowns can be conveniently expressed as $$\Lambda_{k}:=\Lambda_{k}(\{c_{\boldsymbol{t}}\}_{\{\boldsymbol{t}\in \{0,1\}^{K}\}},\Theta_{K}):=\Theta_{K}\cup \left\{\{c_{\boldsymbol{t}}\}_{\{\boldsymbol{t}\in \{0,1\}^{K}\}}:c_{\boldsymbol{t}} \in (0,1)\text{ for all }\boldsymbol{t}\in \{0,1\}^{K}\right\}  $$
Our key parameter of interest is the following function over $\theta \in \Theta_{K}$ defined as $$Q(\theta)=\sum_{\boldsymbol{t} \in \mathcal{T}}c_{t}\mathbb{P}_{\theta}\left(\boldsymbol{T}=\boldsymbol{t}\right)+\lambda f(\theta).$$ where $\mathcal{T}$ denote the set of all possible $2^{K}$ many treatment combinations and $f:\Theta_{K}\to \mathbb{R}$ is a known real valued function and $\lambda>0$ is a known constant. We try to infer about the following subset of the parameter space determined from $Q(\theta)$ precisely stated as, $$\Theta^{*}_{K}:=\arg\max_{ \theta \in \Theta_{K}}Q(\theta). $$
we also denote $\theta^{*}$ to be an arbitrary element of $\Theta^{*}_{K}. $Now, we provide a brief heursitics regarding the various terms appearing in the problem set-up
\begin{itemize}
    \item \textbf{Exploration: }The term $\lambda f(\theta)$ allows for more \textit{exploration}. In this article, we take $f$ to be the entropy function associated with the treatment assignment vector. Hence, $$f(\theta)=-\sum_{i=1}^{k}\left\{\theta_{i}\log (\theta_{i})+(1-\theta_{i})\log(1-\theta_{i}) \right\}, $$ the maximum occurs when each $\theta_{i} \approx 0.5$ allowing the system to \textit{explore more}.
    \item \textbf{Exploitation: }The term $\sum_{\boldsymbol{t} \in \mathcal{T}}c_{t}\mathbb{P}_{\theta}\left(\boldsymbol{T}=\boldsymbol{t}\right)$ allows for more \textit{exploitation}. Maximizing this term would always yield $\theta$ with values concentrated at a single co-ordinate.
\end{itemize}
\subsection{Adaptive estimator}
Our approach proceeds by first proposing an estimator of $Q(\theta)$ denoting it by $\hat{Q}(\theta)$. Before proposing our estimator, we describe some notations. For each $\boldsymbol{t} \in \mathcal{T},$ we define,
\be
\mathcal{I}_{\boldsymbol{t}}&:=\left\{i \in [n]: \boldsymbol{T}_{i}=\boldsymbol{t}\right\},\\
n_{\boldsymbol{t}}&:=\left|\mathcal{I}_{\boldsymbol{t}}\right|,\\
\hat{c}_{\boldsymbol{t}}&:=\frac{\left(\sum_{i \in \mathcal{I}_{t}}Y_i\right) }{n_{\boldsymbol{t}}}I_{\{\mathcal{I}_{\boldsymbol{t}}\neq \phi\}}
\ee
Equipped with this we now define our estimator for $Q(\theta)$ for each $\theta \in \Theta_{K}$ to be, $$\hat{Q}(\theta)=\sum_{\boldsymbol{t} \in \mathcal{T}}\hat{c}_{\boldsymbol{t}} \mathbb{P}_{\theta}\left(\boldsymbol{T}=\boldsymbol{t}\right)+\lambda f(\theta). $$ Our final estimator obtained from $\hat{Q}(\theta)$ is defined as: $$\hat{\theta}_{n}:=\arg\max _{\theta \in \Theta_{K}}\hat{Q}(\theta)$$ We make the following assumption before stating our main theorem regarding consistency of $\{\hat{\theta}_{n}\}_{\{n \in \mathbb{N}\}}.$ The first of which is that we assume $\Theta^{*}_{K}=\{\theta^{*}\}$ although the arguments of the proof proceeds to the general case when $ \Theta^{*}_{K}$ is any arbitrary subset of $\Theta_{K}.$ Secondly, we assume that conditional on $\{\mathcal{I}_{\boldsymbol{t}}\}_{\{\boldsymbol{t} \in \mathcal{T}\}}$, we have $$\{Y_i\}_{\{i \in \mathcal{I}_{\boldsymbol{t}}\}} \stackrel{i.i.d}{\sim} \text{Ber}\left(c_{\boldsymbol{t}}\right) \text{ for all }\boldsymbol{t} \in \mathcal{T}.  $$ Here too we want to emphasize that this assumption is not restrictive as our arguments proceeds over general case when the random variables are uniformly bounded and even when they have sub-exponentially decaying tails with their sub-exponential norms uniformly bounded.
\begin{theorem}
Under the previous assumptions and additionally assuming that for every $\varepsilon > 0$ there exist an $\eta > 0$ ( with $\eta=\Omega\left(\sqrt{\frac{|\mathcal{T}|\log |\mathcal{T}|}{n}}\right)$ ) with, $$\inf_{|\theta- \theta^*|>\varepsilon}\left|Q(\theta)-Q(\theta^*)\right| > \eta. $$ Then, $\hat{\theta}_{n} \stackrel{p}{\to}\theta^{*}$ as $n \to \infty$ if $$\frac{|\mathcal{T}|\log |\mathcal{T}|}{n}=o(1).$$
\end{theorem}
\begin{proof}
We arbitrarily fix a $\delta>0$. We define two events $A$ and $B$ as
\be
A&:=\left\{\sup_{\boldsymbol{t} \in \mathcal{T}}\left|\sqrt{n_{\boldsymbol{t}}}\left(\hat{c}_{\boldsymbol{t}}-c_{\boldsymbol{t}}\right)\right|\leq \sqrt{\frac{(\delta+1)\log|\mathcal{T}|}{2}}\right\},\\
B&:=\left\{\sup_{\boldsymbol{t} \in \mathcal{T}}\left|\frac{n_{\boldsymbol{t}}}{\mathbb{E}[n_{\boldsymbol{t}}]}-1\right|\leq \sqrt{\frac{2(\delta+1)\log |\mathcal{T}|}{\mathbb{E}[n_{\boldsymbol{t}}]}}\right\},
\ee
By a standard application of Hoeffding's inequality ( for instance \textbf{Theorem 2.2.2} in \cite{vershynin2018high}) conditioning on $\{\mathcal{I}_{\boldsymbol{t}}\}_{\{\boldsymbol{t} \in \mathcal{T}\}}$, we obtain that $\mathbb{P}(A)\geq 1-O\left(|\mathcal{T}|^{-\delta}\right).$ We now turn to event $B.$ Under our previous assumption, $n_{\boldsymbol{t}} \sim \text{Bin}\left(n,\frac{1}{|\mathcal{T}|}\right)$ for each $\boldsymbol{t} \in \mathcal{T}$ and hence $\mathbb{E}[n_{\boldsymbol{t}}]=\frac{n}{|\mathcal{T}|}.$ By Bernstein's inequality ( see for instance \cite{bernstein1964modification}), we obtain $\mathbb{P}(B)\geq 1-O\left(|\mathcal{T}|^{-\delta}\right).$ Hence, by an Union bound, we conclude that, $\mathbb{P}\left(A\cap B\right)\geq 1-O\left(|\mathcal{T}|^{-\delta}\right). $ Under the event $B$, for every $\boldsymbol{t}\in \mathcal{T},$
\be
n_{\boldsymbol{t}}& \approx \frac{n}{|\mathcal{T}|}(1+o(1)),\\
& \geq \frac{n}{2|\mathcal{T}|}.
\ee
Hence, working under the event $A \cap B,$ we obtain that for any $\theta \in \Theta_{K}$
\be
|\hat{Q}(\theta)-Q(\theta)|&\leq \sqrt{\frac{(1+\delta)\log |\mathcal{T}|}{2}}\left(\sum_{t \in \mathcal{T}}\frac{\mathbb{P}_{\theta}\left(\boldsymbol{T}=\boldsymbol{t}\right)}{\sqrt{n_{\boldsymbol{t}}}}I_{\{n_{\boldsymbol{t}}\geq 1\}}\right),\\ 
& \leq \sqrt{(1+\delta)\log |\mathcal{T}|}\left(\sum_{t \in \mathcal{T}}\frac{\mathbb{P}_{\theta}\left(\boldsymbol{T}=\boldsymbol{t}\right)}{\sqrt{\frac{n}{|\mathcal{T}|}}}\right),\\
& = \sqrt{\frac{(1+\delta)|\mathcal{T}|\log |\mathcal{T}|}{n}},\\
\implies \sup_{\theta \in \Theta_K}|\hat{Q}(\theta)-Q(\theta)|&\leq \sqrt{\frac{(1+\delta)|\mathcal{T}|\log |\mathcal{T}|}{n}}.
\ee
Fix an $\varepsilon > 0$ and consider the corresponding $\eta > 0$ obtained. For a slowly increasing sequence $\delta$, choose $n$ large enough for which, $$c_{N}:=\sqrt{\frac{(1+\delta)|\mathcal{T}|\log |\mathcal{T}|}{n}} < \frac{\eta}{2}. $$ Hence, the event $|\hat{\theta}_{n}-\theta^{*}|> \varepsilon$ implies
\be
Q(\hat{\theta}_{n})& \leq Q(\theta^{*})-\eta,\\
& \leq \hat{Q}(\theta^{*})+c_{N}-\eta,\\
& \leq \hat{Q}(\hat{\theta}_{n})+c_{N}-\eta,\\
& \leq  \hat{Q}(\hat{\theta}_{n})+c_{N}-2c_{N},\\
\ee
From which we obtain that, 
\be
c_{N}&\leq \hat{Q}(\hat{\theta}_{n})-Q(\hat{\theta}_{n}),\\
&\leq \sup_{\theta \in \Theta_K}|Q(\theta)-\hat{Q}(\theta)|,\\
\ee
This implies that the event $|\hat{\theta}_{n}-\theta^{*}|> \varepsilon$ is a subset of $(A\cap B)^{c}$, from which we obtain, $$\mathbb{P}\left(|\hat{\theta}_{n}-\theta^{*}|< \varepsilon\right) \geq 1-O\left(|\mathcal{T}|^{-\delta}\right).$$ Taking a sequence $\delta$ sufficiently slowly to infinity proves our theorem.
\end{proof}
We now illustrate some remarks that we can deduce from this proof. First of all for a chosen $\lambda > 0$ and $\eta$ as properly defined in the theorem statement, we define the \textit{rate of convergence} for $\hat{\theta}_{n}$ to be
\be
\varepsilon_{\lambda}:=\inf\left\{\varepsilon: \inf_{|\theta- \theta^*|>\varepsilon}\left|Q(\theta)-Q(\theta^*)\right| > \eta\right\}.
\ee
It is very easy to observe from this definition that $\varepsilon_{\lambda}$ is a decreasing function of $\lambda$.
\subsection{Curse of Dimensionality: Restriction on potential outcomes}
In this section we try to provide some formal justification on why the assumption that $|\mathcal{T}|<<n$ is necessary for any meaningful inference on $\theta^{*}$. For simplicity, we assume that the parametric family correspond to that of the multinomial distribution with the parameter space $\Theta_{k}=[0,1]^{K}.$  We denote the set of potential outcomes to be, $$\Gamma:=  \left\{\{c_{\boldsymbol{t}}\}_{\{\boldsymbol{t}\in \{0,1\}^{K}\}}:c_{\boldsymbol{t}} \in (0,1)\text{ for all }\boldsymbol{t}\in \{0,1\}^{K}\right\}. $$ The following theorem provides a mathematical justification to our claim for this special example.
\begin{theorem}
There exist $\varepsilon > 0,$ such that $$\inf_{\hat{\theta}}\sup_{\Tilde{c}\in \Gamma}\mathbb{P}_{\Tilde{c}}\left(|\hat{\theta}-\theta^{*}|>\varepsilon\right)\geq \Omega\left(\exp\left(-\frac{2n}{|\mathcal{T}|}\right)\right), $$ where for each $\Tilde{c}\in \Gamma$ we define  $$\theta^{*}=\arg \max_{\theta \in \Theta_{K}} \left<\Tilde{c}, \mathbb{P}_{\theta}(\boldsymbol{T}=\cdot)\right>.$$
\end{theorem}
\begin{proof}
We denote two treatment combinations, $\boldsymbol{t}_{1}=\{0\}^{K}$ and $\boldsymbol{t}_{2}=\{1\}\times\{0\}^{K-1}.$ For any arbitrary $\beta,\gamma \in \{0,1\},$ we fix an $\Tilde{d}\in \Gamma$ such that $\Tilde{d}=(d_{\boldsymbol{t}}: \boldsymbol{t}\in \{0,1\}^{K})$ with $ d_{\boldsymbol{t}_1}=\beta,d_{\boldsymbol{t}_2}=\gamma$ and $d_{\boldsymbol{t}}\approx 0$ for all $\boldsymbol{t}\in \mathcal{T}\backslash \{\boldsymbol{t}_1,\boldsymbol{t}_2\}.$ We also define the following events,
\be
A&=\left\{\hat{\theta} \in [0,1/2)\times [0,1]^{K-1}\right\},\\
B&=\left\{\mathcal{I}_{\boldsymbol{t}_1}=\mathcal{I}_{\boldsymbol{t}_2}=\phi\right\},\\
\ee
W.L.O.G. we assume that $\mathbb{P}\left(A \cap B\right) \geq \frac{\mathbb{P}\left(B\right)}{2}.$ We note that when $\beta=0,\gamma=1$, the $\theta^{*}$ corresponding to  $\Tilde{d}$ is equal to $(0,1)\times \{0\}^{K-2}.$ Hence, by choosing $\varepsilon=1/10000$, our main argument proceeds as follows
\be
\inf_{\hat{\theta}}\sup_{\Tilde{c}\in \Gamma}\mathbb{P}_{\Tilde{c}}\left(|\hat{\theta}-\theta^{*}|>\varepsilon\right) &\geq \inf_{\hat{\theta}}\sup_{\beta,\gamma \in \{0,1\}} \mathbb{P}_{\Tilde{d}}\left(|\hat{\theta}-\theta^{*}|>\varepsilon\right),\\
&\geq \inf_{\hat{\theta}}\sup_{\beta=0,\gamma=1} \mathbb{P}_{\Tilde{d}}\left(|\hat{\theta}-\theta^{*}|>\varepsilon\right),\\
&\geq \inf_{\hat{\theta}}\sup_{\beta=0,\gamma=1} \mathbb{P}_{\Tilde{d}}\left(|\hat{\theta}-\theta^{*}|>\varepsilon,A\cap B\right),\\
&\geq \mathbb{P}\left(B\right)/2=\Omega\left(\exp\left(-\frac{2n}{|\mathcal{T}|}\right)\right).
\ee
\end{proof}
This proofs our theorem. 
\section{Weighting estimator}
In this section, we deal with the weighting estimator which is defined to be 
\be
\Tilde{Q}(\theta) & =  \frac{1}{n}\sum_{i=1}^nY_i \prod_{k=1}^K
\frac{\mathbb{P}_{\theta}(T_{ik})}{\mathbb{P}(T_{ik})},\\
&=\sum_{\boldsymbol{t}\in \mathcal{T}}\Tilde{c}_{\boldsymbol{t}}\mathbb{P}_{\theta}\left(\boldsymbol{T}=\boldsymbol{t}\right)
\ee
where $ \Tilde{c}_{\boldsymbol{t}}:=\frac{\sum_{i \in \mathcal{I}_{\boldsymbol{t}}}Y_{i}}{n\mathbb{P}(\boldsymbol{T}=\boldsymbol{t})}$ for each $\boldsymbol{t}\in \mathcal{T}$. In our next theorem we compute the expectation and variance of this statistics ( always setting $\lambda=0$).
\begin{theorem}\label{thm:w_var}
We have the following for each $\theta \in \Theta$,
\be
\mathbb{E}[\Tilde{Q}(\theta)]&=Q(\theta),\\
\mathbb{V}[\Tilde{Q}(\theta)]&=\Theta\left(\frac{|\mathcal{T}|\max_{\boldsymbol{t}\in \mathcal{T}} \mathbb{P}_{\theta}\left(\boldsymbol{T}=\boldsymbol{t}\right) }{n}\right).
\ee
\end{theorem}
\begin{proof}
Let $\hat{\mathbb{E}}, \hat{\mathbb{V}}$ denote the conditional expectation and variance conditional on the vector $\left\{I_{\{i \in \mathcal{I}_{\boldsymbol{t}}\}}\right\}_{\{\boldsymbol{t}\in \mathcal{T},i \in [n]\}} $. Through simple computations, we can easily have that for each $\boldsymbol{t}\in \mathcal{T}$,
\be 
\hat{\mathbb{E}}[\Tilde{c}_{\boldsymbol{t}}]&=\frac{c_{\boldsymbol{t}}n_{\boldsymbol{t}}}{n\mathbb{P}\left(\boldsymbol{T}=\boldsymbol{t}\right) },\\
\hat{\mathbb{V}}\left(\Tilde{c}_{\boldsymbol{t}}\right)&=\frac{\sigma^2n_{\boldsymbol{t}} }{(n\mathbb{P}(\boldsymbol{T}=\boldsymbol{t}))^2 }
\ee
Under the additional assumption that $\mathbb{V}[Y_{i}]=\sigma^2$ for all $i$. Taking unconditional expectation again in the first relation provides us that $\mathbb{E}[\Tilde{c}_{\boldsymbol{t}}]=c_{\boldsymbol{t}}$ and from this we easily have that $\mathbb{E}[\Tilde{Q}(\theta)]=Q(\theta). $ Using the second relation we can obtain in a similar way that  
\be
\mathbb{E}[\hat{\mathbb{V}}(\Tilde{Q}(\theta))]&=\frac{\sigma^2|\mathcal{T}|}{n}\sum_{\boldsymbol{t}\in \mathcal{T}}\left(\mathbb{P}_{\theta}(\boldsymbol{T}=\boldsymbol{t})\right)^2,\\
&\leq \frac{\sigma^2|\mathcal{T}|\max_{\boldsymbol{t}\in \mathcal{T}} \mathbb{P}_{\theta}\left(\boldsymbol{T}=\boldsymbol{t}\right)}{n}.
\ee
Now we observe that,
\be
\mathbb{V}(\hat{\mathbb{E}}[\Tilde{Q}(\theta))])&=\mathbb{V}\left(\sum_{\boldsymbol{t}\in \mathcal{T}}\frac{c_{\boldsymbol{t}}n_{\boldsymbol{t}}\mathbb{P}_{\theta}(\boldsymbol{T}=\boldsymbol{t})}{n \mathbb{P}(\boldsymbol{T}=\boldsymbol{t})}\right),\\
& \leq \sum_{\boldsymbol{t}\in \mathcal{T}}\frac{c^{2}_{\boldsymbol{t}}\mathbb{V}(n_{\boldsymbol{t}})(\mathbb{P}_{\theta}(\boldsymbol{T}=\boldsymbol{t}))^2}{(n \mathbb{P}(\boldsymbol{T}=\boldsymbol{t}) )^2},\\
&\leq  \frac{|\mathcal{T}|}{n}\sum_{\boldsymbol{t}\in \mathcal{T}}c^{2}_{\boldsymbol{t}}(\mathbb{P}_{\theta}(\boldsymbol{T}=\boldsymbol{t}))^2,\\
& \leq  \frac{|\mathcal{T}|\max_{\boldsymbol{t}\in \mathcal{T}} \mathbb{P}_{\theta}\left(\boldsymbol{T}=\boldsymbol{t}\right)}{n}\mathbb{E}_{\theta}[c^2_{\boldsymbol{t}}].
\ee
We now use the following relation that,
\be
\mathbb{V}[\Tilde{Q}(\theta)]&=\mathbb{V}(\hat{\mathbb{E}}[\Tilde{Q}(\theta))])+\mathbb{E}[\hat{\mathbb{V}}(\Tilde{Q}(\theta))],\\
&=\Theta\left(\frac{|\mathcal{T}|\max_{\boldsymbol{t}\in \mathcal{T}} \mathbb{P}_{\theta}\left(\boldsymbol{T}=\boldsymbol{t}\right) }{n}\right).
\ee
\end{proof}
Now, we try to obtain lower bound in terms of variance for any \textit{unbiased estimator} of the parameter $Q(\boldsymbol{\theta})$. We work under the assumption that observations coming from distinct treatment combinations are independent and for any treatment combination $\boldsymbol{t} \in \mathcal{T}$, we have $$Y^{(\boldsymbol{t})}_{1}, \dots,Y^{(\boldsymbol{t})}_{n_{\boldsymbol{t}}} \stackrel{i.i.d}{\sim} \mathcal{N}\left(c_{\boldsymbol{t}},\sigma^{2}_{\boldsymbol{t}}\right).$$ We denote $\boldsymbol{c}=(c_{\boldsymbol{t}})_{\boldsymbol{t}\in \mathcal{T}}$ and use this to denote: $$\Psi(\boldsymbol{c})=\sum_{\boldsymbol{t}\in \mathcal{T}}c_{\boldsymbol{t}} \mathbb{P}_{\theta}\left(\boldsymbol{T}=\boldsymbol{t}\right) $$ also we denote $\boldsymbol{Y}$ to denote the set of observations. An equivalent way of stating our assumption is that $Y_{1},\dots, Y_{n}$ are i.i.d samples from a mixture distribution of $|\mathcal{T}|$ many normal distributions with uniform weights i.e, $$Y_{1} \sim \sum_{ \boldsymbol{t}\in \mathcal{T}}\frac{1}{|\mathcal{T}|}\mathcal{N}\left(c_{\boldsymbol{t}},\sigma^{2}_{\boldsymbol{t}}\right) $$  We now move into our next theorem.
\begin{theorem}
Under our previous assumptions and the fact that $\sup_{\boldsymbol{t} \in \mathcal{T}}\left\{|c_{\boldsymbol{t}}|,\sigma_{\boldsymbol{t}}\right\} \leq 1$, for any unbiased estimator $T(\boldsymbol{Y})$ of $Q(\theta)$ we have,
$$\text{Var}\left(T(\boldsymbol{Y})\right) \geq \frac{C|\mathcal{T}|}{n}\sum_{\boldsymbol{t}\in \mathcal{T}}\left(\mathbb{P}_{\boldsymbol{\theta}} \left(\boldsymbol{T}=\boldsymbol{t}\right)\right)^2. $$
Under the additional assumption that $\frac{|\mathcal{T}|}{n}=o(1)$, we obtain from this relation that
$$\frac{\left(\Tilde{Q}(\boldsymbol{\theta})-Q(\boldsymbol{\theta})\right)}{\sqrt{\mathbb{V}[\Tilde{Q}(\theta)]}} \stackrel{d}{\to} \mathcal{N}(0,1). $$
\end{theorem}
\begin{proof}
Try to implement \textit{Cramer-Rao inequality} for proving the lower bound. Now we attempt to provide a \textit{Central Limit Theorem} type result for our weighting estimator. For each $i=1.2.\dots,n$, we denote  $$X^{\theta}_{i}=Y_{i} \prod_{k=1}^K
\frac{\mathbb{P}_{\theta}(T_{ik})}{\mathbb{P}(T_{ik})}.$$ Then $X^{\theta}_{1},\dots, X^{\theta}_{n}$ denote $n$ i.i.d samples with 
\be 
\mathbb{E}[X^{\theta}_{1}]&=\Tilde{Q}(\theta),\\
\mathbb{V}[X^{\theta}_{1}]&=|\mathcal{T}|\sum_{\boldsymbol{t}\in \mathcal{T}}\sigma^{2}_{\boldsymbol{t}}\left(\mathbb{P}_{\theta}(\boldsymbol{T}=\boldsymbol{t})\right)^{2}-\left(\Tilde{Q}(\theta)\right)^{2}\\
&=O\left(|\mathcal{T}|\sum_{\boldsymbol{t}\in \mathcal{T}}\left(\mathbb{P}_{\theta}(\boldsymbol{T}=\boldsymbol{t})\right)^{2}\right),
\ee
Analogously, one can also obtain that,
\be
\mathbb{E}\left[\left|X^{\theta}_{1}-\Tilde{Q}(\theta)\right|^{2+\delta}\right]&=O\left(\mathbb{E}\left[|X|^{2+\delta}\right]\right),\\
&=O\left(|\mathcal{T}|^{1+\delta}\sum_{\boldsymbol{t}\in \mathcal{T}}\left(\mathbb{P}_{\theta}(\boldsymbol{T}=\boldsymbol{t})\right)^{2+\delta}\right)
\ee
Now, we introduce $s^{2}_{n}=\sum_{i=1}^{n}\mathbb{V}[X^{\theta}_{1}]$ which with the help of above relations, help us to conclude that $$s^{2}_{n}=O\left(n|\mathcal{T}|\sum_{\boldsymbol{t}\in \mathcal{T}}\left(\mathbb{P}_{\theta}(\boldsymbol{T}=\boldsymbol{t})\right)^{2}\right) $$
Finally, in order to proof the second part of the statement of the theorem, we confirm the \textbf{Lyapunov's condition} for some $\delta > 0$ for which one needs to obtain
\be
V_{n}&=\frac{\sum_{i=1}^{n} \mathbb{E}[|X^{\theta}_{i}-\Tilde{Q}(\theta)|^{2+\delta}] }{s^{2+\delta}_{n}},\\
&=\frac{n\mathbb{E}\left[|X^{\theta}_{1}-\Tilde{Q}(\theta)|^{2+\delta}\right]}{s^{2+\delta}_{n}},\\
&=\frac{O\left(n|\mathcal{T}|^{1+\delta}\sum_{\boldsymbol{t}\in \mathcal{T}}\left(\mathbb{P}_{\theta}(\boldsymbol{T}=\boldsymbol{t})\right)^{2+\delta}\right)}{O\left(n^{1+\delta/2}|\mathcal{T}|^{1+\delta/2}\left(\sum_{\boldsymbol{t}\in \mathcal{T}}\left(\mathbb{P}_{\theta}(\boldsymbol{T}=\boldsymbol{t})\right)^{2}\right)^{1+\delta/2}\right)},\\
&= O\left(\left\{\frac{|\mathcal{T}|}{n}\right\}^{\delta/2}\right),\\
&=o(1).
\ee
\end{proof}
It would be very helpful to provide a \textit{Uniform Central Limit Theorem} type result for the weighting estimator. As indicated previously, we denote $ \Tilde{c}_{\boldsymbol{t}}:=\frac{\sum_{i \in \mathcal{I}_{\boldsymbol{t}}}Y_{i}}{n\mathbb{P}(\boldsymbol{T}=\boldsymbol{t})}$ for each $\boldsymbol{t}\in \mathcal{T}$. This provides us with the following relation
\be
\Tilde{c}_{\boldsymbol{t}}-c_{\boldsymbol{t}}=\frac{\sum_{i \in \mathcal{I}_{\boldsymbol{t}} }\left(y_{i}-c_{\boldsymbol{t}}\right)   }{n\mathbb{P}\left(\boldsymbol{T}=\boldsymbol{t}\right) }+\frac{\left(n_{\boldsymbol{t}}-n\mathbb{P}\left(\boldsymbol{T}=\boldsymbol{t}\right)\right)c_{\boldsymbol{t}}  }{n\mathbb{P}\left(\boldsymbol{T}=\boldsymbol{t}\right)}
\ee
This provides us with the following relation,
\be
\sqrt{\frac{n}{|\mathcal{T}|}}\left(\Tilde{Q}(\boldsymbol{\theta})-Q(\boldsymbol{\theta})\right)=&\sum_{\boldsymbol{t} \in \mathcal{T}}\sqrt{\frac{n_{\boldsymbol{t}} }{n\mathbb{P}(\boldsymbol{T}=\boldsymbol{t})}}\left(\frac{\sum_{i \in \mathcal{I}_{\boldsymbol{t}} }\left(y_{i}-c_{\boldsymbol{t}}\right)   }{\sqrt{n_ {\boldsymbol{t}} }\sigma_{\boldsymbol{t}} }\right)\sigma_{\boldsymbol{t}}\mathbb{P}_{\boldsymbol{\theta}}\left(\boldsymbol{T}=\boldsymbol{t}\right)\\
&+\sum_{\boldsymbol{t} \in \mathcal{T}}\left(\frac{\left(n_{\boldsymbol{t}}-n\mathbb{P}\left(\boldsymbol{T}=\boldsymbol{t}\right)\right) }{\sqrt{n\mathbb{P}\left(\boldsymbol{T}=\boldsymbol{t}\right)}}\right)c_{\boldsymbol{t}}\mathbb{P}_{\boldsymbol{\theta}}\left(\boldsymbol{T}=\boldsymbol{t}\right)
\ee
For addressing the high dimensional case, we break our analysis into two cases. The first term can be bounded as
\be
\left|\sum_{\boldsymbol{t} \in \mathcal{T}}\sqrt{\frac{n_{\boldsymbol{t}} }{n\mathbb{P}(\boldsymbol{T}=\boldsymbol{t})}}\left(\frac{\sum_{i \in \mathcal{I}_{\boldsymbol{t}} }\left(y_{i}-c_{\boldsymbol{t}}\right)   }{\sqrt{n_ {\boldsymbol{t}} }\sigma_{\boldsymbol{t}} }\right)\sigma_{\boldsymbol{t}}\mathbb{P}_{\boldsymbol{\theta}}\left(\boldsymbol{T}=\boldsymbol{t}\right)\right| & \leq \max_{\boldsymbol{t}}\mathbb{P}_{\theta}(\boldsymbol{T}=\boldsymbol{t})\left|\mathcal{N}\left(0,|\mathcal{T}|\right)\right|\\
& = O_{p}\left(\max_{\boldsymbol{t}}\mathbb{P}_{\theta}(\boldsymbol{T}=\boldsymbol{t})\sqrt{|\mathcal{T}|}\right).
\ee
Now, we deal with the second term. By applying a standard CLT result about the multinomial count vector and then using the High Dimensional version of  Cram\'er-Wold device, we obtain
\be
\left|\sum_{\boldsymbol{t} \in \mathcal{T}}\left(\frac{\left(n_{\boldsymbol{t}}-n\mathbb{P}\left(\boldsymbol{T}=\boldsymbol{t}\right)\right) }{\sqrt{n\mathbb{P}\left(\boldsymbol{T}=\boldsymbol{t}\right)}}\right)c_{\boldsymbol{t}}\mathbb{P}_{\boldsymbol{\theta}}\left(\boldsymbol{T}=\boldsymbol{t}\right)\right| & \leq \left|\mathcal{N}\left(0,\sum_{\boldsymbol{t} \in \mathcal{T}}\left(\mathbb{P}_{\boldsymbol{\theta}}\left(\boldsymbol{T}=\boldsymbol{t}\right)\right)^{2}\right)\right|\\
&=  O_{p}\left(\sqrt{\max_{\boldsymbol{t}}\mathbb{P}_{\theta}(\boldsymbol{T}=\boldsymbol{t})}\right).
\ee
Hence, combining both these quantities we obtain
$$\left|\sqrt{\frac{n}{|\mathcal{T}|}}\left(\Tilde{Q}(\boldsymbol{\theta})-Q(\boldsymbol{\theta})\right)\right|=O_{p}\left(\max_{\boldsymbol{t}}\mathbb{P}_{\theta}(\boldsymbol{T}=\boldsymbol{t})\sqrt{|\mathcal{T}|}\right)+O_{p}\left(\sqrt{\max_{\boldsymbol{t}}\mathbb{P}_{\theta}(\boldsymbol{T}=\boldsymbol{t})}\right) $$
uniformly over $\theta$. Hence, by suitably controlling $\theta$ we can force the RHS to be $o_{p}(1)$. An interesting next step would be to compare the bias and variance with our adaptive estimator. We do this in the following 
\begin{theorem}
We have the following for each $\theta \in \Theta$,
\be
\mathbb{E}[\hat{Q}(\theta)]&=Q(\theta)\left(1-\exp{\left(-\frac{n}{|\mathcal{T}|}\right)}\right),\\
\mathbb{V}[\Tilde{Q}(\theta)]&=\Theta\left(\min\left\{\sigma^2,\frac{\sigma^2|\mathcal{T}|}{n}\right\}\max_{\boldsymbol{t}\in \mathcal{T}} \mathbb{P}_{\theta}\left(\boldsymbol{T}=\boldsymbol{t}\right)+\exp{\left(-\frac{n}{|\mathcal{T}|}\right)}\sum_{\boldsymbol{t}\in \mathcal{T}}c^{2}_{\boldsymbol{t}}(\mathbb{P}_{\theta}(\boldsymbol{T}=\boldsymbol{t}))^2\right).
\ee
\end{theorem}
\begin{proof}

Let $\hat{\mathbb{E}}, \hat{\mathbb{V}}$ denote the conditional expectation and variance conditional on the vector $\left\{I_{\{i \in \mathcal{I}_{\boldsymbol{t}}\}}\right\}_{\{\boldsymbol{t}\in \mathcal{T},i \in [n]\}} $. Through simple computations, we can easily have that for each $\boldsymbol{t}\in \mathcal{T}$,
\be 
\hat{\mathbb{E}}[\hat{c}_{\boldsymbol{t}}]&=c_{\boldsymbol{t}}I_{\{n_{\boldsymbol{t}}\neq 0\}},\\
\hat{\mathbb{V}}\left(\hat{c}_{\boldsymbol{t}}\right)&=\frac{\sigma^2I_{\{n_{\boldsymbol{t}}\neq 0\}} }{n_{\boldsymbol{t}} }
\ee
Under the additional assumption that $\mathbb{V}[Y_{i}]=\sigma^2$ for all $i$. Taking unconditional expectation again in the first relation provides us that $\mathbb{E}[\hat{c}_{\boldsymbol{t}}]=c_{\boldsymbol{t}}\left(1-\exp{\left(-\frac{n}{|\mathcal{T}|}\right)}\right)$ and from this we easily have that $\mathbb{E}[\hat{Q}(\theta)]=Q(\theta)\left(1-\exp{\left(-\frac{n}{|\mathcal{T}|}\right)}\right). $ Using the second relation we can obtain in a similar way that  
\be
\mathbb{E}[\hat{\mathbb{V}}(\Tilde{Q}(\theta))]&=\min\left\{\sigma^2,\frac{\sigma^2|\mathcal{T}|}{n}\right\}\sum_{\boldsymbol{t}\in \mathcal{T}}\left(\mathbb{P}_{\theta}(\boldsymbol{T}=\boldsymbol{t})\right)^2,\\
&\leq \min\left\{\sigma^2,\frac{\sigma^2|\mathcal{T}|}{n}\right\}\max_{\boldsymbol{t}\in \mathcal{T}} \mathbb{P}_{\theta}\left(\boldsymbol{T}=\boldsymbol{t}\right).
\ee
Now we observe that,
\be
\mathbb{V}(\hat{\mathbb{E}}[\Tilde{Q}(\theta))])&=\mathbb{V}\left(\sum_{\boldsymbol{t}\in \mathcal{T}}c_{\boldsymbol{t}}I_{\{n_{\boldsymbol{t}}\neq 0\}}\mathbb{P}_{\theta}(\boldsymbol{T}=\boldsymbol{t})\right),\\
& \leq \sum_{\boldsymbol{t}\in \mathcal{T}}c^{2}_{\boldsymbol{t}}\mathbb{V}\left(I_{\{n_{\boldsymbol{t}}\neq 0\}}\right)(\mathbb{P}_{\theta}(\boldsymbol{T}=\boldsymbol{t}))^2,\\
& \leq \exp{\left(-\frac{n}{|\mathcal{T}|}\right)}\sum_{\boldsymbol{t}\in \mathcal{T}}c^{2}_{\boldsymbol{t}}(\mathbb{P}_{\theta}(\boldsymbol{T}=\boldsymbol{t}))^2,\\
\ee
We now use the following relation that,
\be
\mathbb{V}[\Tilde{Q}(\theta)]&=\mathbb{V}(\hat{\mathbb{E}}[\Tilde{Q}(\theta))])+\mathbb{E}[\hat{\mathbb{V}}(\Tilde{Q}(\theta))],\\
&=\Theta\left(\min\left\{\sigma^2,\frac{\sigma^2|\mathcal{T}|}{n}\right\}\max_{\boldsymbol{t}\in \mathcal{T}} \mathbb{P}_{\theta}\left(\boldsymbol{T}=\boldsymbol{t}\right)+\exp{\left(-\frac{n}{|\mathcal{T}|}\right)}\sum_{\boldsymbol{t}\in \mathcal{T}}c^{2}_{\boldsymbol{t}}(\mathbb{P}_{\theta}(\boldsymbol{T}=\boldsymbol{t}))^2\right).
\ee
\end{proof}
\subsection{Inference on $Q(\theta^{*})$}
According to our definition, $Q(\theta^{*})=\max_{\theta \in \Theta}Q(\theta)$. So, we construct our estimator as follows: we partition our data into $\mathcal{I}_{1}$ and $\mathcal{I}_{2}$. Hence, we have $[n]=\mathcal{I}_{1} \mathbin {\dot {\cup}} \mathcal{I}_{2}$. Let $\hat{Q}^{(i)}(\cdot)$ denote our estimator obtained from the data set $\mathcal{I}_{i}$ for $i=1,2$. Define, $\hat{\theta}^{(1)}=\arg\max_{\theta \in \Theta}\hat{Q}^{(1)}\left(\theta\right)$. Our final estimator $\hat{Q}$ is obtained as follows, $\hat{Q}= \hat{Q}^{(2)}\left(\hat{\theta}^{(1)}\right)$. One can use the results in \cite{zhang2020floodgate} to conclude it's theoretical properties.
\subsection{Hajek Estimator}
In this section, we define a new estimator for $Q(\theta)$. For that, first for each $i=1,2,\dots,n$, we define $W^{\theta}_{i}=\prod_{k=1}^K
\frac{\mathbb{P}_{\theta}(T_{ik})}{\mathbb{P}(T_{ik})}$. Our new estimator $\Bar{Q}(\theta)$ to be $$\Bar{Q}(\theta):=\frac{\sum_{i=1}^{n}Y_{i}W^{\theta}_{i} }{\sum_{i=1}^{n}Y_{i}W^{\theta}_{i}}. $$ Since, this estimator is exactly the \textit{importance sampling estimator} one encounters in Monte Carlo sampling, incorporating their well-known properties we can very well conclude that  
\section{High dimension}
In this section, we attempt to extend the stochastic intervention approach to address the question of when the number of treatment combination is much more than the number of observations i.e. when $|\mathcal{T}| \gg n$. As worked out in \cref{thm:w_var}, the variance of the $\Tilde{Q}(\theta)$ contains the term $\sum_{\boldsymbol{t}\in \mathcal{T}}(\mathbb{P}_{\theta}(\boldsymbol{T}=\boldsymbol{t}))^2 $ which is the decaying as each $\theta_{i} \to 0.5$. Hence, we can target to estimate $\theta$ around a small box around the vector $(0.5, \dots, 0.5)$ ( there are $K$ terms in the vector). Let $\{\varepsilon_{n}\}_{n \in \mathbb{N}}$ be sequence in $(0,1)$ such that for all $\theta \in [0.5(1-\varepsilon_{n}),0.5(1+\varepsilon_{n})]^{K}$ such that $\mathbb{V}[\Tilde{Q}(\theta)] \to 0$ or more appropriately: $$\frac{|\mathcal{T}|(0.5(1 + \varepsilon_{n}))^{K} }{n} \to 0 ,$$
Then, we look at the following optimization problem: $$\hat{\theta}=\arg\max_{\theta \in [0.5(1-\varepsilon_{n}),0.5(1+\varepsilon_{n})]^{K}} \Tilde{Q}(\theta).$$ We parametrize $n=\alpha^{K}$ where $\alpha \gg 1$ and $\nu_{n}=0.5(1+\varepsilon_{n})$ to be chosen such that $$\frac{|\mathcal{T}| }{n}{\nu_{n}}^{K} \leq C$$ where $C$ is a constant independent of $K$ and suitably chosen by the user. We then perform the following optimization, $$\hat{\theta}=\arg\max_{\theta \in [1-\nu_{n},\nu_{n}]^{K}} \Tilde{Q}(\theta).$$
\begin{subsection}{N\"aive Extension}
Our quantity of interest the the value $\theta=\theta^{*}$ at which the following function is maximised 
$$Q(\theta)=\sum_{\boldsymbol{t} \in \mathcal{T}}c_{t}\mathbb{P}_{\theta}\left(\boldsymbol{T}=\boldsymbol{t}\right). $$
We assume there exist some $k < < \min\{|\mathcal{T}|,n\}$ for which there exist $k$ distinct values $d_{1},d_{2},\dots,d_{k} \in (0,1)$ such that $$\{c_{\boldsymbol{t}}\}_{\{\boldsymbol{t}\in \mathcal{T}\}} = \{d_{1},\dots,d_{k}\}. $$
Also, we define the sets $\mathcal{I}_{1},\dots, \mathcal{I}_{k}$ such that for each $i=1,2,\dots,k$ we have $$\mathcal{I}_{i}=\{\boldsymbol{t}\in \mathcal{T}: c_{\boldsymbol{t}}=d_i\}. $$
The model can be formalized as follows
\begin{itemize}
    \item Known parameters: $k$ and other older ones.
    \item Unknown parameters: $(\mathcal{I}_{1},d_1),\dots, (\mathcal{I}_{k},d_k).$
\end{itemize}
Methodology for inference would be: run a $k$-means algorithms with number of clusters evaluated at $k$ and use the output to make inference on $\theta$. Alternately, one can look at the following model and use a standard EM-algorithm to infer on $\pi_i$ and $d_{i} $, $$Y_{i} \stackrel{\text{i.i.d}}{\sim} \sum_{i=1}^{k}\pi_i \text{Ber}(d_{i}) \quad i=1,2,\dots, n $$
Or in case of continuous outcomes
$$\hat{c}_{\boldsymbol{t}} \stackrel{\text{i.i.d}}{\sim} \sum_{i=1}^{k}\pi_i \mathcal{N}(d_{i},\sigma_{i}^2) \quad \boldsymbol{t}\in \mathcal{T}.$$
where $\hat{c}_{\boldsymbol{t}}$ are conditionally unbiased estimators of potential outcome corresponding to treatment combination $\hat{c}_{\boldsymbol{t}}$.
\begin{itemize}
    \item From the output of the EM-algorithm, the $\hat{d}_{i}$'s produced, let $$\hat{j}:=\arg \max \{\hat{d}_{i}\}. $$ 
    \item Use the sample in the $\hat{j}$-th cluster to estimate the probability weights.
\end{itemize}
\end{subsection}

\section{A Different Formulation}
In this section, we look at the following criterion for optimization, $$Q(\boldsymbol{\theta})=\mathbb{E}_{\boldsymbol{\theta}}[c_{\boldsymbol{t}}]-\lambda \text{Var}_{\boldsymbol{\theta}}[c_{\boldsymbol{t}}] $$ where $ \lambda$ is a tuning parameter which is taken to be strictly positive.  We recall our definition that for each $\boldsymbol{t} \in \mathcal{T}$, we have $\Tilde{c}_{\boldsymbol{t}}=\frac{\sum_{i \in \mathcal{I}_{\boldsymbol{t}}}Y_{i}(\boldsymbol{t})  }{n\mathbb{P}\left(\boldsymbol{T}=\boldsymbol{t}\right)}$. Additionally, we also assume that for each $\boldsymbol{t}\in \mathcal{T}$, we have $Y_{i}(\boldsymbol{t})$ has mean $c_{\boldsymbol{t}}$ and variance $1$ which yields
\be
\mathbb{E}[\Tilde{c}_{\boldsymbol{t}}]=c_{\boldsymbol{t}}, \quad \text{Var}[\Tilde{c}_{\boldsymbol{t}}]=\frac{1}{n\mathbb{P}\left(\boldsymbol{T}=\boldsymbol{t}\right)}
\ee
This motivates us to define
\be
\hat{\text{Var}}_{\boldsymbol{\theta}}[\Tilde{c}_{\boldsymbol{t}}]&=\mathbb{E}_{\boldsymbol{\theta}}\left(\Tilde{c}_{\boldsymbol{t}}-\mathbb{E}_{\boldsymbol{\theta}}[\Tilde{c}_{\boldsymbol{t}}]\right)^2,\\
&=\mathbb{E}_{\boldsymbol{\theta}}[\Tilde{c}^{2}_{\boldsymbol{t}}]-\left(\mathbb{E}_{\boldsymbol{\theta}}[\Tilde{c}_{\boldsymbol{t}}]\right)^2
\ee
Taking $\mathbb{E}$ both sides of the above definition we obtain
\be
\mathbb{E}\left[\hat{\text{Var}}_{\boldsymbol{\theta}}[\Tilde{c}_{\boldsymbol{t}}]\right]=&\mathbb{E}_{\theta}\left[c^{2}_{\boldsymbol{t}}+\frac{1}{n\mathbb{P}\left(\boldsymbol{T}=\boldsymbol{t}\right)}\right]\\
&-\left(\sum_{\boldsymbol{t} \in \mathcal{T}}\left(c^{2}_{\boldsymbol{t}}+\frac{1}{n\mathbb{P}\left(\boldsymbol{T}=\boldsymbol{t}\right)}\right)\left(\mathbb{P}_{\boldsymbol{\theta}}\left(\boldsymbol{T}=\boldsymbol{t}\right)\right)^2+\sum_{\boldsymbol{t}\neq \boldsymbol{s} \in \mathcal{T}}\left\{c_{\boldsymbol{t}}c_{\boldsymbol{s}}\frac{n(n-1)}{n^2}\right\}\mathbb{P}_{\boldsymbol{\theta}}\left(\boldsymbol{T}=\boldsymbol{t}\right)\mathbb{P}_{\boldsymbol{\theta}}\left(\boldsymbol{T}=\boldsymbol{s}\right)\right)\\
=& \text{Var}_{\boldsymbol{\theta}}\left[c_{\boldsymbol{t}}\right]+\Delta_{\theta}+\text{O}\left(\frac{1}{n}\right).
\ee
where $\Delta_{\theta}=\frac{1}{n\mathbb{P}\left(\boldsymbol{T}=\boldsymbol{t}\right)}-\sum_{\boldsymbol{t} \in \mathcal{T}}\frac{\left(\mathbb{P}_{\boldsymbol{\theta}}\left(\boldsymbol{T}=\boldsymbol{t}\right)\right)^2}{n\mathbb{P}\left(\boldsymbol{T}=\boldsymbol{t}\right)}$. Hence, our final candidate estimator for $Q(\boldsymbol{\theta})$ be defined as $$\hat{Q}(\boldsymbol{\theta})=\mathbb{E}_{\boldsymbol{\theta}}[\Tilde{c}_{\boldsymbol{t}}]-\lambda\left(\hat{\text{Var}}_{\boldsymbol{\theta}}[\Tilde{c}_{\boldsymbol{t}}]-\Delta_{\theta}\right) $$
This estimator can be modified to address the case when variance associated to the treatment combination $\boldsymbol{t}\in \mathcal{T}$. We also define other estimators which might be preferable to the previous one in specific circumstances. Here, we list them as follows:
\be
\hat{\text{Var}}^{(1)}_{\boldsymbol{\theta}}[\Tilde{c}_{\boldsymbol{t}}]&=\mathbb{E}_{\boldsymbol{\theta}}[\Tilde{c}^{2}_{\boldsymbol{t}}]-\frac{1}{n(n-1)}\sum_{1\leq i \neq j \leq n}Y_{i}Y_{j}\frac{\mathbb{P}_{\boldsymbol{\theta}}\left(\boldsymbol{T}_{i}=\boldsymbol{t}_{i}, \boldsymbol{T}_{j}=\boldsymbol{t}_{j}\right) }{\mathbb{P}\left(\boldsymbol{T}_{i}=\boldsymbol{t}_{i}, \boldsymbol{T}_{j}=\boldsymbol{t}_{j}\right)}
\ee
Based on this, our final estimator can be stated as 
\be
\hat{Q}^{(1)}(\boldsymbol{\theta})=\mathbb{E}_{\boldsymbol{\theta}}[\Tilde{c}_{\boldsymbol{t}}]-\lambda\left(\hat{\text{Var}}^{(1)}_{\boldsymbol{\theta}}[\Tilde{c}_{\boldsymbol{t}}]-\Delta^{(1)}_{\theta}\right) 
\ee
where $\Delta^{(1)}_{\theta}=\mathbb{E}_{\boldsymbol{\theta}}\left[\frac{\sigma^{2}_{\boldsymbol{t}} }{n\mathbb{P}\left(\boldsymbol{T}=\boldsymbol{t} \right)}\right]$. Analogously, we define another estimator as follows:
\be
\hat{\text{Var}}^{(2)}_{\boldsymbol{\theta}}[\Tilde{c}_{\boldsymbol{t}}]&=\frac{1}{n}\sum_{1 \leq i \leq n}Y^{2}_{i}\frac{\mathbb{P}_{\boldsymbol{\theta}}\left(\boldsymbol{T}_{i}=\boldsymbol{t}_{i}\right) }{\mathbb{P}\left(\boldsymbol{T}_{i}=\boldsymbol{t}_{i}\right)}-\frac{1}{n(n-1)}\sum_{1\leq i \neq j \leq n}Y_{i}Y_{j}\frac{\mathbb{P}_{\boldsymbol{\theta}}\left(\boldsymbol{T}_{i}=\boldsymbol{t}_{i}, \boldsymbol{T}_{j}=\boldsymbol{t}_{j}\right) }{\mathbb{P}\left(\boldsymbol{T}_{i}=\boldsymbol{t}_{i}, \boldsymbol{T}_{j}=\boldsymbol{t}_{j}\right)}
\ee
and based on this we define the following estimator
\be
\hat{Q}^{(2)}(\boldsymbol{\theta})=\mathbb{E}_{\boldsymbol{\theta}}[\Tilde{c}_{\boldsymbol{t}}]-\lambda\left(\hat{\text{Var}}^{(2)}_{\boldsymbol{\theta}}[\Tilde{c}_{\boldsymbol{t}}]-\Delta^{(2)}_{\theta}\right) 
\ee
where $\Delta^{(2)}_{\theta}=\mathbb{E}_{\boldsymbol{\theta}}\left[\sigma^{2}_{\boldsymbol{t}} \right]$.
\section{Basis Selection: A Stochastic Optimization Overview}
Consider the problem where we have data $(Y_{1},X_{1}),\dots,(Y_{n},X_{n})$ where $Y$'s are the response and $X$'s are the $d$-dimensional co-variate having a density $f$ in $[0,1]^{d}$. For simplicity, we assume that $f$ is known and taken to be the uniform distribution over $[0,1]^{d}$. We further assume that 
$$\mathbb{E}[Y|X]=b(X) $$ where $b: [0,1]^{d} \to \mathbb{R}$ is an unknown function with further constraints to be stated shortly. Now, let us take $\{\phi_{1}(\cdot),\dots,\phi_{p}(\cdot)\}$ to be a collection of $p$ \textit{localized bases} in $[0,1]^{d}$ where $p \gg n$. Having introduced this, we let $b$ to satisfy: $$b(x):=\sum_{j=1}^{p}\left<b,\phi_{j}\right>\phi_{j}(x), $$which provides us with the following relation
\be
\|b\|^{2}_{2,f}&=\int b^{2}(x)f(x)dx,\\
&=\sum_{i=1}^{p}|\left<b,\phi_{i}\right>_{f}|^{2},\\
&=\sum_{i=1}^{p}c_{i}
\ee
where for each $i \in \{1,\dotsm,p\}$, we define $c_{i}=|\left<b,\phi_{i}\right>_{f}|^{2}$. The goal is find a subset of $k \ll n$ elements $S$ from $\{\phi_{1}(\cdot),\dots,\phi_{p}(\cdot)\}$ such that $\sum_{j \in S} c_{j}$ is \textit{closest} to $\sum_{i = 1}^{p} c_{i}$. The motivation for this problem can be found in \cite{mukherjee2017semiparametric}, \cite{robins2017higher}, \cite{robins2017minimax}. An \textit{unbiased estimator} of $c_{q}$ is $\hat{c}_{q}$ being defined as $$\hat{c}_{q}=\frac{1}{n(n-1)}\sum_{1\leq i \neq j \leq n}Y_{i}\phi_{q}(X_i)\phi_{q}(X_{j})Y_{j},$$ for each $p=1,2,\dots,n$. We now implement the \textit{stochastic intervention framework} to approach the problem. Let $\alpha_{\theta}=(\alpha_{\theta}(1),\dots,\alpha_{\theta}(p)) $ be an \textit{interpret-able} probability distribution over the integers $\{1,2,\dots,p\}$ parametrized by a very \textit{low dimensional subset} $\theta \in \Theta \subseteq \mathbb{R}^{l}$, i.e, $\sum_{i=1}^{p}\alpha_{\theta}(i)=1$. Hence, our task is the following: find ${\theta}^{*}$ such that $${\theta}^{*}=\arg\max_{\theta}\left\{\sum_{i=1}^{p}c_{i}\alpha_{\theta}(i)+\lambda \text{pen}(\theta)\right\}.$$ for an appropriately chosen penalty function $\text{pen}(\theta)$. In order to make inference on ${\theta}^{*}$ we define
$$\hat{\theta}=\arg\max_{\theta}\left\{\sum_{i=1}^{p}\hat{c}_{i}\alpha_{\theta}(i)+\lambda \text{pen}(\theta)\right\}.$$
Then, our task is to design a meaningful statistical procedure to extract information from the vector $\hat{\theta}$ in order to get those $k$ elements of the basis. This would be our next goal. In order to operationalize this entire procedure we encode the integers $\{1,2,\dots,p\}$ in their binary representation. Hence, each integer can be represented as a string of $\log_{2}(p)$ many $0$'s and $1$'s. Now the probability distribution $\alpha_{\theta}$ can be constructed by considering an independent coin for each position in the bit representation i.e we pick a point from the set $\{1,2,\dots,p\}$ according to $\alpha_{\theta}$ through it's bit-wise representation by flipping $\log_{2}p$ many coins with the $i$-th coin having success probability $\theta_{i}$. Hence, under this sampling scheme, the parameter space $\Theta$ is a $\log_{2}(p)$ dimensional subset ( or $l= \log_{2}(p)$). We consider the penalty function to be the entropy function,
$$\text{pen}(\theta)=-\sum_{i=1}^{l}\left\{\theta_{i}\log (\theta_{i})+(1-\theta_{i})\log(1-\theta_{i}) \right\}, $$
Finally, we choose $k$ elements from the basis by first obtaining $k$ independent draws from $\alpha_{\hat{\theta}}$, i.e., $$i_{1},\dots,i_{k} \stackrel{i.i.d}{\sim} \alpha_{\hat{\theta}} $$ and then returning $\{\phi_{i_1},\dots,\phi_{i_k}\}$.
 \phantomsection
  \addcontentsline{toc}{section}{References}
  \bibliographystyle{amsalpha}
  \bibliography{biblio.bib}

@book{vershynin2018high,
  title={High-dimensional probability: An introduction with applications in data science},
  author={Vershynin, Roman},
  volume={47},
  year={2018},
  publisher={Cambridge university press}
}

@article{bernstein1964modification,
  title={On a modification of Chebyshev’s inequality and on the error in Laplace formula},
  author={Bernstein, SN},
  journal={Collected Works, Izd-vo’Nauka’, Moscow (in Russian)},
  volume={4},
  pages={71--80},
  year={1964}
}

@article{mukherjee2017semiparametric,
  title={Semiparametric efficient empirical higher order influence function estimators},
  author={Mukherjee, Rajarshi and Newey, Whitney K and Robins, James M},
  journal={arXiv preprint arXiv:1705.07577},
  year={2017}
}

@article{robins2017higher,
  title={Higher order estimating equations for high-dimensional models},
  author={Robins, James and Li, Lingling and Mukherjee, Rajarshi and Tchetgen, Eric Tchetgen and van der Vaart, Aad},
  journal={Annals of statistics},
  volume={45},
  number={5},
  pages={1951},
  year={2017},
  publisher={NIH Public Access}
}

@article{robins2017minimax,
  title={Minimax estimation of a functional on a structured high-dimensional model},
  author={Robins, James M and Li, Lingling and Mukherjee, Rajarshi and Tchetgen, Eric Tchetgen and van der Vaart, Aad and others},
  journal={Annals of Statistics},
  volume={45},
  number={5},
  pages={1951--1987},
  year={2017},
  publisher={Institute of Mathematical Statistics}
}

@article{zhang2020floodgate,
  title={Floodgate: inference for model-free variable importance},
  author={Zhang, Lu and Janson, Lucas},
  journal={arXiv preprint arXiv:2007.01283},
  year={2020}
}

\end{document}